\documentclass[12pt]{article}

\usepackage[a4paper]{geometry}
\usepackage{amssymb}
\usepackage{amsmath}
\usepackage{amsthm}
\usepackage[square,comma,numbers]{natbib}
\usepackage{color}

\theoremstyle{definition}
\newtheorem{theorem}{Theorem}[section]
\newtheorem{proposition}[theorem]{Proposition}

\newtheorem{definition}[theorem]{Definition}

\newtheorem{assumption}[theorem]{Assumption}
\newtheorem*{definition*}{Definition}
\newtheorem*{example*}{Example}
\newtheorem*{problem*}{Problem}
\newtheorem*{problems*}{Problems}

\theoremstyle{remark}
\newtheorem{remark}[theorem]{Remark}
\newtheorem*{remark*}{Remark}
\newtheorem{remarks}[theorem]{Remarks}
\newtheorem*{remarks*}{Remarks}

\numberwithin{equation}{section}

\title{Propagation of space-time singularities for perturbed harmonic oscillators}

\author{Kenichi {\scshape Ito}\footnote{Department of Mathematics, Graduate School of Science, Kobe University,
1-1, Rokkodai, Nada-ku, Kobe 657-8501, Japan.
E-mail: \texttt{ito-ken@math.kobe-u.ac.jp}. 
}
\ \& 
Tomoya {\scshape Tagawa}\footnote{Graduate School of Mathematical Sciences, 
The University of Tokyo, 3-8-1 Komaba, Meguro-ku, Tokyo 153-8914, Japan.
E-mail: \texttt{tagawa-tomoya1212@g.ecc.u-tokyo.ac.jp}. 
}
}

\date{}

\begin{document}
\allowdisplaybreaks
\maketitle

\begin{abstract}
We discuss propagation of space-time singularities for the 
quantum harmonic oscillator with time-dependent metric and potential perturbations. 
Reformulating the \textit{quasi-homogeneous wave front set} according to Lascar (1977)
in a semiclassical manner, 
we obtain a characterization of its appearance in comparison with the unperturbed system.
The idea of our proof is based on the argument of Nakamura (2009), 
which was originally devised for the analysis of \textit{spatial} singularities of the Schr\"odinger equation, 
however, the application is non-trivial since the time is no more a parameter, but takes a part in the base variables.
\end{abstract}


\medskip



\tableofcontents

\section{Settings and results}

\subsection{Perturbed harmonic oscillator}

Let $\mathbb R^{1+d}=\mathbb R_t\times \mathbb R^d_x$ with $d\in\mathbb N=\{1,2,\ldots\}$. 
In this paper we investigate space-time singularities of a solution $u\in \mathcal S'(\mathbb R^{1+d})$ to 
the Schr\"odinger equation 
\begin{equation}
\tfrac{\partial}{\partial t}u=-\mathrm iHu
,\quad 
u(0,\cdot)=\phi\in \mathcal S'(\mathbb R^d)
,
\label{25012311}
\end{equation}
with $H$ being the Schr\"odinger operator with a perturbed harmonic potential
\begin{equation}
H=\tfrac12 p_i a_{ij}(t,x)p_j+\tfrac12 |x|^{2}+V(t,x) 
.
\label{25020420}
\end{equation}
Here $ p_i=-\mathrm i\partial/\partial x_i$ for $i=1,\dots,d$, 
and the \textit{Einstein summation convention} is adopted without tensorial superscripts. 
We assume that the perturbation is time-dependent, and of \textit{short-range} 
type as follows. Let $\mathbb N_0=\{0\}\cup\mathbb N$. 
\begin{assumption}\label{250124}
Let $a_{ij},V\in C^\infty(\mathbb R^{1+d};\mathbb R)$ for $i,j=1,\dots,d$, and assume the following.
\begin{enumerate}
\item
For each $(t,x)\in\mathbb R^{1+d}$ the matrix $(a_{ij}(t,x))_{i,j=1,\dots,d}$ is symmetric and positive definite. 
\item
There exists $\epsilon>0$ such that 
for any $\widetilde\alpha=(\alpha_0,\dots,\alpha_d)=(\alpha_0,\alpha)\in\mathbb N_0\times \mathbb N_0^d$ 
there exists $C>0$ such that for any $i,j=1,\dots,d$ and $(t,x)\in \mathbb R^{1+d}$
\[
|\partial^{\widetilde\alpha} (a_{ij}(t,x)-\delta_{ij})|\le C\langle x\rangle^{-1-|\alpha|-\epsilon}
,\quad 
|\partial^{\widetilde\alpha} V(t,x)|\le C\langle x\rangle^{1-|\alpha|-\epsilon}
,
\]
where $\delta$ is the Kronecker delta, and $\langle x\rangle=(1+|x|^{2})^{1/2}$. 
\end{enumerate}
\end{assumption}
\begin{remark}
The above assumption corresponds to the so-called short-range condition 
for the harmonic oscillator in the high-energy regime.
It actually suffices to assume these estimates locally in time, 
but for simplicity we let them be global. 
\end{remark}

Under Assumption~\ref{250124} 
the unique solvability of the Cauchy problem \eqref{25012311} 
is well-known for some more restrictive initial data. 

\begin{theorem}[{\cite[Theorem~6 and Remark~(a)]{Y}}]
Suppose Assumption~\ref{250124}, and set 
\[
\mathcal D_\pm
=\bigl\{\phi \in\mathcal S'(\mathbb R^d);\ (1+x^2+p^2)^{\pm 1/2}\phi\in L^2(\mathbb R^d)\bigr\}
,
\] 
respectively. 
Then there exists a unique family $\{U(t,s)\}_{t,s\in\mathbb R}$ of unitary operators on $L^2(\mathbb R^d)$ with the following properties.
\begin{enumerate}
\item
For any $t,s,r\in\mathbb R$, $U(t,s)U(s,r)=U(t,r)$.
\item
The mapping 
$\mathbb R^2\to\mathcal L(\mathcal D_+),\ (t,s)\mapsto U(t,s)$ is strongly continuous, where $\mathcal L(\mathcal D_+)$ denotes the space of bounded linear operators on $\mathcal D_+$. 
\item
For any $\phi\in \mathcal D_+$ the mapping 
$\mathbb R^2\to\mathcal D_-,\  (t,s)\mapsto U(t,s)\phi$ is continuously differentiable with partial derivatives
\[
\tfrac{\partial}{\partial t}U(t,s)\phi=-\mathrm iHU(t,s)\phi,\quad 
\tfrac{\partial}{\partial s}U(t,s)\phi=\mathrm iU(t,s)H\phi
.
\]
\end{enumerate}
\end{theorem}

Throughout the paper we consider $U(t,s)$ only for $s=0$, and 
thus denote $U(t)=U(t,0)$ for short. 
Now for any $\phi\in\mathcal D_+$ we are going to discuss the singularities of 
$u\in C_{\mathrm b}(\mathbb R;L^2(\mathbb R^d))\subset \mathcal S'(\mathbb R^{1+d})$
given by 
\begin{equation}
u(t,x)=(U(t)\phi)(x)
. 
\label{25012917}
\end{equation}
To be more precise, we characterize them in terms of 
those of the unperturbed solution 
$u_{\mathrm{os}}\in C_{\mathrm b}(\mathbb R;L^2(\mathbb R^d))\subset \mathcal S'(\mathbb R^{1+d})$ 
given by 
\begin{equation}
u_{\mathrm{os}}(t,x)=(\mathrm e^{-\mathrm itH_{\mathrm{os}}}\phi)(x)
,
\quad 
H_{\mathrm{os}}=\tfrac12p^2+\tfrac{1}{2}|x|^2=-\tfrac12\Delta +\tfrac{1}{2}|x|^2. 
\label{250204}
\end{equation}

\subsection{Quasi-homogeneous wave front set}

For the analysis of space-time singularities of the Schr\"odinger equation \eqref{25012311} 
it is more natural and appropriate to introduce the \textit{quasi-homogeneous wave front set} following Lascar~\cite{L}.
Recall that 
for any $n\in\mathbb N$
the \textit{Weyl quantization} of a symbol $a\in C^\infty_{\mathrm c}(\mathbb R^{2n})$ 
is defined through the formula 
\begin{align*}
&
a^{\mathrm W}(z,p_z)v(z)
=(2\pi)^{-n}\int_{\mathbb R^{2n}}\mathrm e^{\mathrm i(z-w)\zeta}
a\bigl(\tfrac12(z+w),\zeta\bigr)v(w)\,\mathrm dw\mathrm d\zeta
, 
\quad 
v\in\mathcal S'(\mathbb R^{n})
.
\end{align*}

\begin{definition}
For any $v\in \mathcal S'(\mathbb R^{1+d})$ define the 
\textit{quasi-homogeneous wave front set of order} $\theta\in (0,\infty)$: 
\[\mathop{\mathrm{WF}}\nolimits^{\mathrm{qh}}_\theta (v)
\subset \mathbb R^{1+d}\times \mathbb S^{d}\] 
as the complement of the set of all $(s,y,\sigma,\eta)\in\mathbb R^{1+d}\times \mathbb S^d$ 
such that 
there exists $a\in C^\infty_{\mathrm c}(\mathbb R^{2(1+d)})$  satisfying $a(s,y,\sigma,\eta)\neq 0$ and 
\[
\|a^{\mathrm W}(t,x,h^\theta p_t,hp_x)v\|_{L^2_{t,x}}=\mathcal O(h^\infty) \quad \text{as }h\to +0
.
\]
\end{definition}

It is clear that $\mathop{\mathrm{WF}}\nolimits^{\mathrm{qh}}_1 (v)$ 
coincides with the ordinary wave front set $\mathop{\mathrm{WF}}(v)$. 
However, $\mathop{\mathrm{WF}}\nolimits^{\mathrm{qh}}_\theta (v)$ for $\theta>1$ 
refines $\mathop{\mathrm{WF}}(u)$ at the 
north and the south poles of the cosphere bundle $S^*\mathbb R^{1+d}\cong \mathbb R^{1+d}\times \mathbb S^{d}$,
while it degrades the rest down to the equator. 
Let us state it more precisely. Set 
\begin{align*}
\mathcal P=\{(t,x,\pm 1,0)\in \mathbb R^{1+d}\times \mathbb S^d\}
,\quad
\mathcal E&=\{(t,x,0,\xi)\in \mathbb R^{1+d}\times \mathbb S^d\}
.
\end{align*}

\begin{proposition}\label{250129}
For any $v\in \mathcal S'(\mathbb R^{1+d})$ and $0<\rho<\theta<\infty$ 
the following holds. 
\begin{enumerate}
\item
If $\mathop{\mathrm{WF}}\nolimits^{\mathrm{qh}}_{\theta} (v)\cap \mathcal E^c\neq\emptyset$, 
then $\mathop{\mathrm{WF}}\nolimits^{\mathrm{qh}}_\rho (v)\cap \mathcal P\neq\emptyset$. 
\item
If $\mathop{\mathrm{WF}}\nolimits^{\mathrm{qh}}_\rho (v)\cap\mathcal P^c\neq\emptyset$, 
then $\mathop{\mathrm{WF}}\nolimits^{\mathrm{qh}}_{\theta} (v)\cap \mathcal E\neq\emptyset$. 
\end{enumerate}
\end{proposition}
\begin{remarks}\label{250207}
\begin{enumerate}
\item
These assertions exactly say that 
$\mathop{\mathrm{WF}}\nolimits^{\mathrm{qh}}_{\theta} (v)$ refines 
$\mathop{\mathrm{WF}}\nolimits^{\mathrm{qh}}_\rho (v)$ at $\mathcal P$,
and 
$\mathop{\mathrm{WF}}\nolimits^{\mathrm{qh}}_\rho (v)$ refines $\mathop{\mathrm{WF}}\nolimits^{\mathrm{qh}}_{\theta} (v)$ at $\mathcal E$. 
In the terminology of Melrose~\cite{M}, $\mathop{\mathrm{WF}}\nolimits^{\mathrm{qh}}_{\theta} (v)$ is simultaneously a \textit{blow-up} and a \textit{blow-down} 
of $\mathop{\mathrm{WF}}\nolimits^{\mathrm{qh}}_\rho (v)$ at $\mathcal P$ and $\mathcal E$, respectively. 
\item\label{25020715}
We could further refine the assertion 1 in terms of the northern/southern hemisphere and the north/south pole, respectively, 
and the assertion 2 in terms of longitudes, but we omit them.  
\end{enumerate}
\end{remarks}

The proof of Proposition~\ref{250129} is not difficult, and we omit it. 

We clearly see that the principal part of the Schr\"odinger equation \eqref{25012311} 
is \textit{quasi-homogeneous} in the $t$- and the $x$-derivatives, 
and we should investigate the quasi-homogeneous wave front set of order $\theta=2$
for a solution $u$. In fact, $\theta=2$ is the critical order in the following sense.
\begin{theorem}\label{25012918}
Suppose Assumption~\ref{250124}, 
let $\phi\in \mathcal D_+$, and let $u\in\mathcal S'(\mathbb R^{1+d})$ be given by \eqref{25012917}.
Then 
\[
\mathop{\mathrm{WF}}\nolimits^{\mathrm{qh}}_\theta (u)
\subset
\begin{cases} 
\mathcal P& \text{if }\theta\in (0,2),\\ 
\bigl\{
\bigl(t,x,-\tfrac12\mu^2a_{ij}(t,x)\xi_i\xi_j,\mu\xi\bigr)
\in \mathbb R^{1+d}\times \mathbb S^d\bigr\}&\text{if }\theta=2,\\ 
\mathcal E &\text{if }\theta\in (2,\infty),
\end{cases}
\]
where $\mu=\mu(t,x,\xi)$ is a positive normalization factor satisfying 
\begin{equation}
\tfrac14\mu^4(a_{ij}(t,x)\xi_i\xi_j)^2+\mu^{2}|\xi|^2=1.
\label{25122717}
\end{equation}
\end{theorem}

For the proof of Theorem~\ref{25012918} 
we have only to reformulate the \textit{microlocal ellipticity} 
associated with a semiclassical quantization 
$a^{\mathrm W}(t,x,h^\theta p_t,hp_x)$ for each given $\theta>0$, 
and to construct a \textit{parametrix} in the standard manner. 
See, e.g., a textbook by Martinez \cite{Mar}. Let us omit the details.

\subsection{Classical high-energy asymptotics}

The propagation of singularities for \eqref{25012311} is described by 
the high-energy limit of the corresponding classical mechanics with time-dependent 
classical Hamiltonian
\[K(t,x,\xi)=\tfrac12a_{ij}(t,x)\xi_i\xi_j+\tfrac12|x|^2+V(t,x),\quad (t,x,\xi)\in\mathbb R^{2(1+d)}. \]
However, due to our short-range assumption 
we can drop the potential and even the time-dependence, 
so that it in fact reduces to a classical Hamiltonian  
\begin{equation}
K_s(x,\xi)=\tfrac12a_{ij}(s,x)\xi_i\xi_j,\quad (x,\xi)\in\mathbb R^{2d},
\label{251227}
\end{equation}
for fixed $s\in\mathbb R$. Let us take a closer look at \eqref{251227}. 
Denote by 
\begin{equation}
(x(t),\xi(t))=(x(t;s,y,\eta),\xi(t;s,y,\eta))
\label{250620}
\end{equation}
a solution to the Hamilton equations
\begin{equation}
  \dot x_i=a_{ij}(s,x)\xi_j,\ \  
\dot\xi_i=-\tfrac12(\partial_ia_{jk}(s,x))\xi_j\xi_k
\quad \text{for }i=1,\dots,d
\label{202511302324}
\
\end{equation}
with initial condition $(x(s),\xi(s))=(y,\eta)\in\mathbb R^{2d}$.

\begin{definition}\label{25080316}
Initial data $(s,y,\eta)\in \mathbb R^{1+2d}$ is said to be \textit{non-trapping} if  
\[
\lim_{\lambda\to\infty} |x(0;s,y,\lambda\eta)|=\infty
.
\]
In addition, we denote by $\Omega\subset \mathbb R^{1+2d}$ the set of all non-trapping initial data.
\end{definition}

Our characterization of the quasi-homogeneous wave front set 
for \eqref{25012917} involves the following classical scattering data in the high-energy limit. 

\begin{proposition}[{\cite[Lemma 3]{N2}}]\label{250123}
Suppose Assumption~\ref{250124}.
Then $\Omega\subset \mathbb R^{1+2d}$ is an open subset. 
Moreover, for any $(s,y,\eta)\in \Omega$ there exist the limits 
\begin{align*}
x_+(s,y,\eta)&:=\lim_{\lambda\to\infty}(x(0,s,y,\lambda\eta)+s\xi(0,s,y,\lambda\eta)),
\\
\xi_+(s,y,\eta)&:=\lim_{\lambda\to\infty}\lambda^{-1}\xi(0,s,y,\lambda\eta)
.
\end{align*}
\end{proposition}

\subsection{Main result}

Now we present our main result 
that characterizes the quasi-homogeneous wave front set of a solution $u$ to \eqref{25012311}. 
Note that by Theorem~\ref{25012918} we may only consider the critical order $\theta=2$. 
Moreover, as in Theorem~\ref{25012918} again, we can parametrize 
all possible points of $\mathop{\mathrm{WF}}\nolimits^{\mathrm{qh}}_2(u)$ 
by $(t,x,\xi)\in \mathbb R^{1+d}\times (\mathbb R^d\setminus\{0\})$, setting  
\[
\Pi(t,x,\xi)=
\bigl(t,x, -\tfrac12\mu^2a_{ij}(t,x)\xi_i\xi_j,\mu\xi\bigr)
\in \mathbb R^{1+d}\times \mathbb S^d.\]
Here $\mu$ is the positive normalization factor  
satisfying \eqref{25122717}, or 
\[
\mu=\mu(t,x,\xi)=\Bigl(\sqrt{2}\left(\bigl[\xi^4+(a_{ij}(t,x)\xi_i\xi_j)^2\bigr]^{1/2}+\xi^2\right)\Bigr)^{-1/2}
.
\]
In particular, for the unperturbed case we denote it by    
\[
\Pi_{\mathrm{os}}(t,x,\xi)=\Bigl(t,x, -\bigl(2^{1/2}-1\bigr),\bigl(2^{3/2}-2\bigr)^{1/2}\xi\big/|\xi|\Bigr)
\in \mathbb R^{1+d}\times \mathbb S^d
.\]

\begin{theorem}\label{250208}
Suppose Assumption~\ref{250124}.
For any $\phi\in\mathcal D_+$ let 
$u,u_{\mathrm{os}}\in \mathcal S'(\mathbb R^{1+d})$ be from 
\eqref{25012917}, \eqref{250204}, respectively,
and for any $(s,y,\eta)\in \Omega$ with $ s \in (-\pi,\pi)$ let 
\[x_+=x_+(s,y,\eta),\quad \xi_+=\xi_+(s,y,\eta)\]
be from Proposition~\ref{250123}.
Then one has 
\[
\Pi(s,y,\eta)\in \mathop{\mathrm{WF}}\nolimits^{\mathrm{qh}}_2(u)
\quad
\text{if and only if} 
\quad
\Pi_{\mathrm{os}}(s,x_+,\xi_+)\in\mathop{\mathrm{WF}}\nolimits^{\mathrm{qh}}_2(u_{\mathrm{os}}).
\]
\end{theorem}
\begin{remarks}
\begin{enumerate}
\item
Theorem~\ref{250208} may be seen as a characterization of 
$\mathop{\mathrm{WF}}\nolimits^{\mathrm{qh}}_2(u)$ in terms of the initial state $\phi$, 
since $u_{\mathrm{os}}$ has an explicit integral representation 
involving $\phi$ and the Mehler kernel. 
This essentially differs from the result by Lascar~\cite{L}, which discussed two points
on a bicharacteristic curve at the same time $s\in\mathbb R$.

\item
For the Schr\"odinger equation with a decaying potential, 
Fujii--Ito~\cite{FI} obtained a third equivalent condition 
that is more directly written by $\phi$, without going through the free propagator. 
In addition, using this condition, 
they reproduced the result of Lascar~\cite{L}. 
We could not verify the corresponding results for a perturbed harmonic oscillator. 
We also refer to a recent work by Gell-Redman--Gomes--Hassell~\cite{GGH} for a relevant result
on space-time singularities. 

\end{enumerate}
\end{remarks}

%

  
%

Propagation of singularities for the Schr\"odinger equation 
is different from the wave equation in that it has infinite propagation speed. 
Thus the standard method for the wave equation does not work for the Schr\"odinger equation. 
Lascar~\cite{L} actually developed a class of pseudodifferential operators that suited 
PDEs with quasi-homogeneous principal parts, 
however for the Schr\"odinger equation 
he could only compare the space-time singularities of the same time component.

After Lascar 
the main focus has shifted to \textit{spatial} singularities, or those of time-slices, of a solution.
Among others, for the Schr\"odinger equation with a decaying potential, 
complete characterizations of spatial singularities were given by 
Hassell--Wunsch~\cite{HW} and Nakamura~\cite{N2}. 
The method of Nakamura~\cite{N2} was applied to the harmonic oscillator with short-range perturbations
by Mao--Nakamura~\cite{MN}, 
and further with long-range perturbations by Mao~\cite{Mao}. 
As for singularities for the harmonic oscillators, see also Wunsch~\cite{W2} for the trace formula, 
Doi~\cite{Doi} for perturbations of linear growth, 
Rodino--Trapasso~\cite{RT} for the Gabor wave front set, 
and Ito--Kato~\cite{IK} for the wave packet transform.

The present paper focuses back on \textit{space-time} singularities 
for the harmonic oscillator. 
We are directly motivated by an ongoing project due to Fujii--Ito~\cite{FI} which deals with 
a decaying potential. 
Since Lascar~\cite{L}, 
we are not aware of any other works investigating the space-time singularities 
for the Schr\"odinger equation, except for a recent work by 
Gell-Redman--Gomes--Hassell~\cite{GGH}. 
Our arguments at last boil down to those similar to \cite{N2,MN}, 
however there are some non-trivial difficulty before the reduction.
See the discussion at the end of Section~\ref{202512081453}.

Finally, we remark that Fujii--Ito~\cite{FI} obtain more direct description of 
space-time singularities in terms of the initial state. 
They employ a variant of the wave front set,
which in some sense refines the \textit{homogeneous wave front set} by Nakamura~\cite{N1}, 
or the \textit{quadratic scattering wave front set} by Wunsch~\cite{W1}.
As for such these variants of the wave front set we refer to Ito~\cite{I}, Fukushima~\cite{F} 
and Rodino--Trapasso~\cite{RT}.

\subsection{Strategy of the proof}\label{202512081453}

Finally we close this section with strategy of the proof for Theorem \ref{250208},
which would motivate the arguments of the following sections.

Fix any non-trapping initial data $(s,y,\eta)\in\Omega$,
and let $x_+=x_+(s,y,\eta)$ and $\xi_+=\xi_+(s,y,\eta)$ be from Proposition~\ref{250123}. 
Assume $\Pi(s,y,\eta)\notin \mathop{\mathrm{WF}}^{\mathrm{qh}}_2(u)$. 
Then by definition there exists $a_0\in C^\infty_{\mathrm c}(\mathbb R^{2(1+d)})$ such that  
$a_0\bigl(s,y,-\tfrac12a_{ij}(s,y)\eta_i\eta_j,\eta\bigr)\neq 0$, and that 
\[
\|a_0^{\mathrm W}(t,x,h^2 p_t,hp_x)u\|_{L^2_{t,x}}=\mathcal O(h^\infty)\ \ \text{as }h\to +0
.
\]
Thus, if we can construct $a_1\in C^\infty_{\mathrm c}(\mathbb R^{2(1+d)})$ such that 
$a_1\bigl(s,x_+,-\tfrac12|\xi_+|^2,\xi_+\bigr)\neq 0$, and that 
\begin{align}
\begin{split}
&\|a_1^{\mathrm W}(t,x,h^2 p_t,hp_x)u_{\mathrm{os}}\|_{L^2_{t,x}}^2
\\&
\le 
\|a_0^{\mathrm W}(t,x,h^2 p_t,hp_x)u\|_{L^2_{t,x}}^2
+\mathcal O(h^\infty) \ \ \text{as }h\to +0,
\end{split}
\label{251228}
\end{align}
then we obtain $\Pi_{\mathrm{os}}(s,x_+,\xi_+)
\not\in\mathop{\mathrm{WF}}\nolimits^{\mathrm{qh}}_2(u_{\mathrm{os}})$, 
and the proof is done. The converse is proved by the same manner. 

In either way, we have to ``connect'' two operators 
$|a_0^{\mathrm W}(t,x,h^2 p_t,hp_x)|^2$ 
and $|a_1^{\mathrm W}(t,x,h^2 p_t,hp_x)|^2$ 
with the desired properties.
For that we are going to interpolate them with an operator-valued function 
\begin{equation}
B(\kappa)=b^{\mathrm W}(\kappa, t,x,h^{2}p_t,h p_x)
.
\label{25062214}
\end{equation}
Now, let us set  
\begin{equation}
I(\kappa)
=
\bigl\langle \mathrm e^{-\mathrm i\kappa tH_{\mathrm{os}}}U((1-\kappa)t)\phi,
B(\kappa)
\mathrm e^{-\mathrm i\kappa tH_{\mathrm{os}}}U((1-\kappa)t)\phi\bigr\rangle_{L^2_{t,x}}
,
\label{25122816}
\end{equation}
and solve 
\[
\tfrac{\mathrm d}{\mathrm d\kappa}I(\kappa)=\mathcal O(h^\infty)
\ \ \text{uniformly in }\kappa\in [0,1]. 
\]
By direct computations we can write 
\begin{align*}
\tfrac{\mathrm d}{\mathrm d\kappa}I(\kappa)
&=
\bigl\langle \mathrm e^{-\mathrm i\kappa tH_{\mathrm{os}}}U((1-\kappa)t)\phi,
\mathbf D B(\kappa)
\mathrm e^{-\mathrm i\kappa tH_{\mathrm{os}}}U((1-\kappa)t)\phi\bigr\rangle_{L^2_{t,x}}
\end{align*}
with 
\begin{align*}
\mathbf D B(\kappa)
&=
\tfrac{\mathrm d}{\mathrm d\kappa}B(\kappa)+\mathrm i[L(\kappa),B(\kappa)]
,\\
L(\kappa)&=-t
\bigl\{\mathrm e^{-\mathrm{i} \kappa tH_{\mathrm{os}}}H((1-\kappa)t)\mathrm e^{\mathrm{i} \kappa tH_{\mathrm{os}}}-H_{\mathrm{os}}\bigr\},
\end{align*}
where $H(t)$ is an operator defined by the Hamiltonian $H$ with $t$ fixed.
Hence it reduces to the equation
\begin{equation}
\tfrac{\mathrm d}{\mathrm d\kappa}B(\kappa)+\mathrm i[L(\kappa),B(\kappa)]
=\mathcal O(h^\infty)\ \ \text{as }h\to +0,
\label{25042517}
\end{equation}
and we will construct a symbol $b(\kappa,t,x,\tau,\xi)$ of $B(\kappa)$
as an asymptotic sum. 
We note that the operator $L(\kappa)$ has an exact symbol
\begin{equation}
\begin{aligned}
l(\kappa,t,x,\tau,\xi) &=
-\tfrac12 t \bigl\{a_{ij}((1-\kappa)t,\cos(-\kappa t)x+ \sin(- \kappa t)\xi )-\delta_{ij}\bigr\} \\
&\quad \cdot(-\sin(-\kappa t)x_{i}+\cos(-\kappa t)\xi_{i})(-\sin(-\kappa t)x_{j}+\cos(-\kappa t)\xi_{j}) \\
&\quad -tV((1-\kappa)t,\cos(-\kappa t)x+ \sin(- \kappa t)\xi)
,\label{250619}
\end{aligned}
\end{equation}
cf.\ \cite{MN}, 
and we are led to study the Hamiltonian flow associated with $l$.

Although we will further reduce it to a flow similar to Mao--Nakamura~\cite{MN}, 
we emphasize that our reduction procedure is quite different from theirs. 
In the analysis of spatial singularities 
the time $t$ is an \textit{external} parameter, and we can directly use it to connect two operators.
However, as for the space-time singularities, 
$t$ is involved in integrations as a base variable,
and we can no more use it as a parameter. 
Thus we have to introduce a new extra parameter $\kappa$ as in \eqref{25122816}, 
which we consider is non-trivial.
In addition, our potential is time-dependent, so that the analysis of the classical mechanics
gets more intricate.

\section{Classical mechanics}

In this section we study the asymptotics of a solution to the Hamilton equations 
for the Hamiltonian $l$ from \eqref{250619}. 
More precisely, we consider 
\begin{align}
\tfrac{\mathrm d}{\mathrm d\kappa}t&=0
,
\label{25042320}
\\ 
\begin{split}
\tfrac{\mathrm d}{\mathrm d\kappa} x_i&
=
-t (a_{ij}((1-\kappa)t,\cos(-\kappa t)x + \sin(-\kappa t) \xi)-\delta_{ij})
\\
&\qquad{}  \cos(-\kappa t)(-\sin(-\kappa t)x_{j}+\cos(-\kappa t)\xi_{j}) \\
&\quad{}
-\tfrac12 t (\partial_ia_{jk})((1-\kappa)t,\cos(-\kappa t)x + \sin(-\kappa t) \xi)\sin(-\kappa t)\\
&\qquad (-\sin(-\kappa t)x_{j}+\cos(-\kappa t)\xi_{j})(-\sin(-\kappa t)x_{k}+\cos(-\kappa t)\xi_{k}) \\
&\quad - t \sin(-\kappa t) (\partial_iV)((1-\kappa)t,\cos(-\kappa t)x + \sin(-\kappa t))
,
\end{split}
\label{25042321}
\\
\begin{split}
\tfrac{\mathrm d}{\mathrm d\kappa}\tau&
= \tfrac12  (a_{ij}((1-\kappa)t,\cos(-\kappa t)x + \sin(-\kappa t)\xi)-\delta_{ij}) \\ 
&\qquad (-\sin(-\kappa t)x_{i}+\cos(-\kappa t)\xi_{i})(-\sin(-\kappa t)x_{j}+\cos(-\kappa t)\xi_{j})\\
&\quad +\tfrac12 (1-\kappa)t(\partial_ta_{ij})((1-\kappa)t,\cos(-\kappa t)x + \sin(-\kappa t)\xi) \\
&\qquad (-\sin(-\kappa t)x_{i}+\cos(-\kappa t)\xi_{i})(-\sin(-\kappa t)x_{j}+\cos(-\kappa t)\xi_{j}) \\
&\quad - \tfrac{1}{2} t \kappa (\partial_ka_{ij})((1-\kappa)t,\cos(-\kappa t)x + \sin(-\kappa t)\xi) \\
&\qquad (-\sin(-\kappa t)x_{k}+\cos(-\kappa t)\xi_{k})(-\sin(-\kappa t)x_{i}+\cos(-\kappa t)\xi_{i}) \\
&\qquad (-\sin(-\kappa t)x_{j}+\cos(-\kappa t)\xi_{j}) \\
&\quad -t\kappa (a_{ij}((1-\kappa)t,\cos(-\kappa t)x + \sin(-\kappa t)\xi)-\delta_{ij}) \\
&\qquad (-\cos(-\kappa t)x_{i}+\sin(-\kappa t)\xi_{i}) (-\sin(-\kappa t)x_{j}+\cos(-\kappa t)\xi_{j}) \\
&\quad +V((1-\kappa)t,\cos(-\kappa t)x + \sin(-\kappa t)\xi)\\
&\quad +(1-\kappa)t(\partial_tV)((1-\kappa)t,\cos(-\kappa t)x + \sin(-\kappa t)\xi) \\
&\quad -\kappa t(\partial_iV)((1-\kappa)t,\cos(-\kappa t)x + \sin(-\kappa t)\xi) \\
&\qquad (-\sin(-\kappa t)x_{i}+\cos(-\kappa t)\xi_{i})
,
\end{split}
\label{25042322}
\\
\begin{split}
  \tfrac{\mathrm d}{\mathrm d\kappa}\xi_i
&=
\tfrac12 t ((\partial_ia_{jk})((1-\kappa)t,\cos(-\kappa t)x + \sin(-\kappa t)\xi)) \cos(-\kappa t)\\
&\qquad (-\sin(-\kappa t)x_{j}+\cos(-\kappa t)\xi_{j})(-\sin(-\kappa t)x_{k}+\cos(-\kappa t)\xi_{k})\\
&\quad +t(a_{ij}((1-\kappa)t,\cos(-\kappa t)x + \sin(-\kappa t) \xi)-\delta_{ij})\\
&\qquad (-\sin(-\kappa t ))(-\sin(-\kappa t)x_{j}+\cos(-\kappa t)\xi_{j}) \\
&\quad  +t(\partial_iV)((1-\kappa)t,\cos(-\kappa t)x + \sin(-\kappa t)\xi) \cos(-\kappa t)
.
\end{split}
\label{25042323}
\end{align}
We solve the equations \eqref{25042320}--\eqref{25042323} with initial data 
\begin{equation}
(t(0),x(0),\tau(0),\xi(0))=(s,y,\lambda^2\sigma,\lambda\eta),
\label{25042324}
\end{equation}
and investigate the limit of a solution as $\lambda \to\infty$. 

\subsection{Reduction to simpler Hamilton equations}

Let us first reduce the equations \eqref{25042320}--\eqref{25042323} to simpler ones. 
It is trivial from \eqref{25042320} and \eqref{25042324} that $t\equiv s$.
Substitute it into \eqref{25042321}--\eqref{25042323}, 
and change the dependent variables as 
\begin{align*}
z(\kappa)&=\cos(-\kappa s )x(\kappa)+ \sin(-\kappa s)\xi(\kappa),
\\
\gamma(\kappa)&=-\sin(-\kappa s)x(\kappa) +\cos(-\kappa s)\xi(k),
\\
\rho(\kappa)&=\tau(\kappa)
+(1-\kappa)(( \tfrac12 a_{ij}((1-\kappa)s,z(\kappa))\gamma_{i}(\kappa)\gamma_{j}(\kappa)+V((1-\kappa)s,z(\kappa))) \\
&\quad +\tfrac{\kappa}{2}|\gamma|^{2}+\tfrac{1}{2}|z|^{2}
.
\end{align*}
Then we actually obtain 
\begin{align}
\tfrac{\mathrm d }{\mathrm d\kappa}z_i
&=
-sa_{ij}((1-\kappa)s,z)\gamma_j
,
\label{25061917}
\\
\tfrac{\mathrm d}{\mathrm d\kappa}\gamma_i
&=
\tfrac12 s (\partial_ia_{jk}((1-\kappa)s,z))\gamma_j\gamma_k
+s (\partial_iV)((1-\kappa)s,z)+sz_{i}(\kappa)
,
\label{25061918}
\\
\tfrac{\mathrm d }{\mathrm d\kappa}\rho
&=
0
.
\label{25061919}
\end{align}
The equation \eqref{25061919} for $\rho$ is trivially solved, and we can obtain the following expression for $\tau$:
\begin{equation}
  \begin{split}\label{2511261553}
    \tau (\kappa) &= \lambda^{2}\left(\sigma +\tfrac{1}{2} a_{ij}(s,y)\eta_{i}\eta_{j}-\tfrac{\kappa}{2} \left|\tfrac{\gamma(\kappa)}{\lambda}\right|^{2} \right) \\
              &\quad-(1-\kappa)( \tfrac12 a_{ij}((1-\kappa)s,z(\kappa))\gamma_{i}(\kappa)\gamma_{j}(\kappa)+V((1-\kappa)s,z(\kappa))) \\  
              &\quad -\tfrac{1}{2}|z(\kappa)|^{2}+V(s,y)+\tfrac{1}{2}|y|^{2}. 
  \end{split}
\end{equation}
Thus it suffices to consider the equations \eqref{25061917} and \eqref{25061918} for $(z,\gamma)$.

\subsection{Classical Mourre-type estimates}

Let us denote by 
\[
(z(\kappa),\gamma(\kappa))=(z(\kappa,s,y,\lambda\eta),\gamma(\kappa,s,y,\lambda\eta))
\]
a solution to the equations \eqref{25061917} and \eqref{25061918} 
with initial data 
\[
(z(0),\gamma(0))=(y,\lambda\eta),
\]
and we investigate its asymptotic behavior as $\lambda\to\infty$. 
For that we further set 
\[
(z_\lambda(\kappa,s,y,\eta),\gamma_\lambda(\kappa,s,y,\eta))
=(z(\lambda^{-1} \kappa,s ,y,\lambda\eta),\lambda^{-1}\gamma(\lambda^{-1}\kappa,s,y,\lambda\eta))
.
\]
Then they satisfy
\begin{align}
\tfrac{\mathrm d }{\mathrm d\kappa}z_{\lambda,i}
&=
-s a_{ij}((1-\lambda^{-1}\kappa)s,z_\lambda)\gamma_{\lambda,j},
\label{2506191918}
\\
\begin{split}
\tfrac{\mathrm d}{\mathrm d\kappa}\gamma_{\lambda,i}
&=
\tfrac12 s (\partial_i a_{jk}((1-\lambda^{-1}\kappa)s,z_{\lambda}))
  \gamma_{\lambda,j}\gamma_{\lambda,k} \\
&\quad 
+ s\lambda^{-2} (\partial_i V)((1-\lambda^{-1}\kappa)s,z_{\lambda})
+ s\lambda^{-2} z_{\lambda,i}(\kappa)
\end{split}
\label{2506191919}
\end{align}
with 
\[
(z_\lambda(0,s,y,\eta),\gamma_\lambda(0,s,y,\eta))=(y,\eta)
.
\]
The equations \eqref{2506191918} and \eqref{2506191919} are obviously the 
the Hamilton equations for the Hamiltonian
\[
H_\lambda(\kappa,s,z,\gamma)= - \tfrac12  s a_{ij}((1-\lambda^{-1}\kappa)s,z)\gamma_i\gamma_j
-\tfrac{1}{2}s\lambda^{-2}|z|^{2} -s\lambda^{-2}V((1-\lambda^{-1}\kappa)s,z).
\]
The next proposition claims that, for $\lambda >0$ sufficiently large, 
the non-trapping condition implies that $z_{\lambda}$ remains away from the origin 
for relatively small $\kappa$.
\begin{proposition}\label{250804}
Let $(s,y,\eta)$  be non-trapping in the sense of
Definition~\ref{25080316}. Then there exist $0<\delta<1$, $\lambda_{0}\ge 1$, $c_{1},c_{2}>0$, and a neighborhood $\widetilde{\Omega} \subset \Omega$ of
$(s,y,\eta)$ such that for any $\lambda\ge \lambda_{0}$,
$\kappa\in[0,\delta\lambda]$, and $(\tilde{s},\tilde{y},\tilde{\eta})
\in \widetilde{\Omega}$
\[
  \bigl| z_{\lambda}(\kappa,\tilde{s},\tilde{y},\tilde{\eta}) \bigr|
  \;\ge\; c_{1}\kappa - c_{2}.
\]
\end{proposition}
\begin{proof}
Throughout the proof, we use $C_{\star}>0$ to denote a generic constant
independent of $\tilde{s}$ and $\lambda$.
We note that, for the purpose of the proof, we may assume, for the time being, that
$|\tilde{s}-s|<1$, $|\tilde{y}-y|<1$, $|\tilde{\eta}-\eta|<1$
and further restrictions on $\tilde{s},\tilde{y},\tilde{\eta}$ will be imposed later as needed.

\smallskip
\noindent
\textit{Step 1.}\ 
We first deduce a rough kinetic energy estimate. 
Let us differentiate 
\begin{align*}
\begin{split}
&\tfrac{\mathrm d}{\mathrm d\kappa}\bigl(a_{ij}((1- \lambda^{-1}\kappa)\tilde{s},z_\lambda)\gamma_{\lambda,i}\gamma_{\lambda,j} + \lambda^{-2} |z_{\lambda}|^{2}\bigr)
\\&=
-\tilde{s}\lambda^{-1}(\partial_ta_{ij})((1-\lambda^{-1}\kappa)\tilde{s},z_\lambda)\gamma_{\lambda,i}\gamma_{\lambda,j} \\
&\quad + a_{ij}((1-\lambda^{-1}\kappa)\tilde{s},z_\lambda)(2\tilde{s}\lambda^{-1}\partial_{i} V) (\lambda^{-1}\gamma_{\lambda,j}).
\end{split}
\label{25062016}
\end{align*}
By the Cauchy--Schwarz inequality this implies 
\begin{equation*}
  \begin{split}
    &\bigl|\tfrac{\mathrm d}{\mathrm d\kappa}\bigl(a_{ij}((1-\lambda^{-1}\kappa)\tilde{s},z_\lambda)\gamma_{\lambda,i}\gamma_{\lambda,j} + \lambda^{-2}|z_{\lambda}|^{2}\bigr)\bigr|  
     \\ 
     &\le 
       C_1 \lambda^{-1}\bigl(a_{ij}((1-\lambda^{-1}\kappa)\tilde{s},z_\lambda)\gamma_{\lambda,i}\gamma_{\lambda,j} + \lambda^{-2} |z_{\lambda}|^{2} \bigr)
      +C_1\lambda^{-2},  
  \end{split}
\end{equation*}
so that using the Gronwall inequality we have
\begin{equation*}
 \begin{split}
  &a_{ij}((1-\lambda^{-1}\kappa)\tilde{s},z_\lambda)\gamma_{\lambda,i}\gamma_{\lambda,j} +\lambda^{-2}|z_{\lambda}|^{2} \\ 
  &\quad \ge 
  \mathrm e^{-C_{1} \lambda^{-1}\kappa}\bigl(a_{ij}(\tilde{s},\tilde{y})\tilde{\eta}_{i}\tilde{\eta}_{j} +\lambda^{-2}|\tilde{y}|^{2} +\lambda^{-1}\bigr)-\lambda^{-1}, \\[6pt]
 &a_{ij}((1-\lambda^{-1}\kappa)\tilde{s},z_\lambda)\gamma_{\lambda,i}\gamma_{\lambda,j}+ \lambda^{-2}|z_{\lambda}|^{2} \\ 
 &\quad \le  \mathrm e^{C_1\lambda^{-1}\kappa}\bigl(a_{ij}(\tilde{s},\tilde{y})\tilde{\eta}_{i}\tilde{\eta}_{j}+\lambda^{-2}|\tilde{y}|^{2}+\lambda^{-1}\bigr) -\lambda^{-1}.
 \end{split}  
\end{equation*}
Hence by letting $\lambda_0\ge 1$ be large enough and choosing $c_{1},c_{2}>0$ appropriately,
it follows that for any $\lambda\ge \lambda_0$ and $\kappa\in [0,\lambda]$
\begin{equation}
0<c_1\le a_{ij}((1-\lambda^{-1}\kappa)\tilde{s},z_\lambda)\gamma_{\lambda,i}\gamma_{\lambda,j} + \lambda^{-2}|z_{\lambda}|^2\le c_2<\infty.
\label{25061920}
\end{equation}
We remark that due to \eqref{2506191918} and \eqref{25061920} we in particular have 
for any $\lambda\ge \lambda_0$ and $\kappa\in [0,\lambda]$
\begin{equation}
    |z_\lambda(\kappa,\tilde{s},\tilde{y},\tilde{\eta})|\le C_2 \kappa , \quad |\gamma_{\lambda}(\kappa,\tilde{s},\tilde{y},\tilde{\eta})| \le C_{2}.
\label{2506192030}
\end{equation}

\smallskip
\noindent
\textit{Step 2.}\ 
We next deduce the classical Mourre-type estimate. 
We differentiate 
\begin{align*}
\begin{split}
\tfrac{\mathrm d^2}{\mathrm d\kappa^2}|z_\lambda(\kappa,\tilde{s},\tilde{y},\tilde{\eta})|^2
&=2\lambda\tfrac{\mathrm d}{\mathrm d\kappa}((- \tilde{s})a_{ij}((1-\lambda^{-1}\kappa)\tilde{s},z_\lambda)z_{\lambda,i}\gamma_{\lambda,j})
\\
&=2\tilde{s}\lambda^{-1} (\partial_{t}a_{ij})((1-\lambda^{-1}\kappa)\tilde{s},z_\lambda) \gamma_{\lambda,i}\gamma_{\lambda,j}\\
&\quad +2\tilde{s}^{2} (\partial_{k}a_{ij})((1-\lambda^{-1}\kappa)\tilde{s},z_\lambda)a_{kl}((1-\lambda^{-1}\kappa)\tilde{s},z_\lambda)\gamma_{\lambda,i}\gamma_{\lambda,l}z_{\lambda,j} \\
&\quad +2\tilde{s}^{2} a_{ij}((1-\lambda^{-1}\kappa)\tilde{s},z_\lambda)(a_{ik}((1-\lambda^{-1}\kappa)\tilde{s},z_\lambda)-\delta_{ik})\gamma_{\lambda,j}\gamma_{\lambda,k} \\
&\quad +2\tilde{s}^{2} a_{ij}((1-\lambda^{-1}\kappa)\tilde{s},z_\lambda)\gamma_{\lambda,i}\gamma_{\lambda,j} \\
&\quad -\tilde{s}^{2} a_{ij}((1-\lambda^{-1}\kappa)\tilde{s},z_\lambda)\partial_{j}a_{kl}((1-\lambda^{-1}\kappa)\tilde{s},z_\lambda) \gamma_{\lambda,k} \gamma_{\lambda,l} z_{\lambda,i} \\
&\quad -2\tilde{s}\lambda^{-2} a_{ij}((1-\lambda^{-1}\kappa)\tilde{s},z_\lambda) z_{\lambda,i} (\partial_{j}V) \\
&\quad -2\tilde{s}\lambda^{-2} a_{ij}((1-\lambda^{-1}\kappa)\tilde{s},z_\lambda) z_{\lambda,i}z_{\lambda,j}.
\end{split}
\label{25062017}
\end{align*}
Then, using the Cauchy--Schwarz inequality and \eqref{2506192030} together with Assumption \ref{250124},
and retaking $\lambda_0\ge 1$ larger if necessary, we obtain 
\begin{equation}
\tfrac{\mathrm d^2}{\mathrm d\kappa^2}|z_\lambda(\kappa,\tilde{s},\tilde{y},\tilde{\eta})|^2
\ge 2\tilde{s}^{2}c_1-C_3\langle z_\lambda\rangle^{-1-\epsilon} - C_{3} \lambda^{-2}|z_{\lambda}|^{2}
.
\label{25062012}
\end{equation}
Hence, by letting $0<\delta_{1}<1$ be sufficiently small,
for any $\kappa \in [0,\delta_{1}\lambda]$, we obtain from \eqref{2506192030} and \eqref{25062012} that
\begin{equation}
    \tfrac{\mathrm d^2}{\mathrm d\kappa^2}|z_\lambda(\kappa,\tilde{s},\tilde{y},\tilde{\eta})|^2
  \ge c_{3} - C_3 \langle z_\lambda \rangle^{-1-\epsilon}
  \label{202511302345}
\end{equation}
with some constant $c_{3}>0$.

\smallskip
\noindent
\textit{Step 3.}\ 
Here we prove that, letting $\lambda_0\ge 1$ be even larger if necessary,
we can deduce that for any $\lambda\ge \lambda_0$ and $\kappa\in [0,\delta \lambda]$,
and for any $(\tilde{s},\tilde{y},\tilde{\eta})$ in a suitably chosen neighborhood $\tilde{\Omega}$ of $(s,y,\eta)$, we have
\begin{equation*}
\bigl| z_{\lambda}(\kappa,\tilde{s},\tilde{y},\tilde{\eta}) \bigr|
  \;\ge\; c_{4}\kappa - C_{4},
\label{2506192045}
\end{equation*}
with some constants $c_{4},C_{4}>0$.
To this end, recall that $(s,y,\eta)$ satisfies the non-trapping condition, and
let $(x(\kappa,s,y,\eta),\xi(\kappa,s,y,\eta))$ be the solution to
\eqref{250620}. 
Then, at $\kappa=0$, the initial data of
$
(x((1-\kappa)s,s,y,\eta), \xi((1-\kappa)s,s,y,\eta))
$
coincides with that of
$(z_\lambda(\kappa), \gamma_\lambda(\kappa))$.
Furthermore, the trajectories $(x(\kappa),\xi(\kappa))$ and
$(z_\lambda(\kappa),\gamma_\lambda(\kappa))$ satisfy the Hamiltonian systems
\eqref{202511302324} and \eqref{2506191918}--\eqref{2506191919}, respectively.
Consequently, using the continuity of solutions to ODEs with respect to the parameters $\lambda$, $s$ and the initial values $y,\eta$,
we can choose $\lambda_{0} \ge 1$ sufficiently large, a sufficiently small neighborhood $\tilde{\Omega}$ of $(s,y,\eta)$, 
and $\kappa_{0} >0$ so that the following inequality hold
\begin{equation}
  |z_{\lambda}(\kappa_{0},\tilde{s},\tilde{y},\tilde{\eta})|^{1+\epsilon} \ge (\tfrac{1}{2} c_{3}C_{3}^{-1})^{-1},
  \quad (\tfrac{\mathrm d}{\mathrm d\kappa} |z_{\lambda}|)(\kappa_{0},\tilde{s},\tilde{y},\tilde{\eta}) >0
  \label{202511302333}
\end{equation}
for any $\lambda \ge \lambda_{0}$, $(\tilde{s},\tilde{y},\tilde{\eta}) \in \tilde{\Omega}$.
Hence, by \eqref{202511302345} and \eqref{202511302333}, we obtain 
\begin{equation*}
    \tfrac{\mathrm d^2}{\mathrm d\kappa^2}|z_\lambda(\kappa,\tilde{s},\tilde{y},\tilde{\eta})|^2
    \ge \tfrac{1}{2}c_{3} > 0
    \label{202511302341}
\end{equation*}
for all $\lambda \ge \lambda_{0}$, $\kappa \in [\kappa_{0},\delta\lambda]$, and 
$(\tilde{s},\tilde{y},\tilde{\eta}) \in \tilde{\Omega}$.
Thus by a standard convexity argument, we obtain the assertion.
\end{proof}
Next we investigate the asymptotic behavior of $(x(\kappa),\xi(\kappa))$
as $\lambda \to \infty$.
Recalling the definitions of $z(\kappa)$, $\gamma(\kappa)$,
$z_{\lambda}(\kappa)$, and $\gamma_{\lambda}(\kappa)$, and using the
scaling relation, we remark that we can obtain
\begin{equation} \label{202512071953}
\begin{aligned}
  (x(\kappa),\xi(\kappa)) &= 
    \exp( s\kappa\, H_{\mathrm{os}} )
    \bigl( z_{\lambda}(\lambda\kappa), \lambda \gamma_{\lambda}(\lambda\kappa) \bigr) \\
  &= \Bigl(
       \mathrm{pr}_{1}\Bigl[
         \exp( s\kappa\lambda\, H_{\mathrm{\mathrm{os},\lambda}} )
         \bigl( z_{\lambda}(\lambda\kappa), \gamma_{\lambda}(\lambda\kappa) \bigr)
       \Bigr], \\[1ex]
       &\quad
       \lambda\, \mathrm{pr}_{2}\Bigl[
         \exp( s\kappa\lambda\, H_{\mathrm{\mathrm{os},\lambda}} )
         \bigl( z_{\lambda}(\lambda\kappa), \gamma_{\lambda}(\lambda\kappa) \bigr)
       \Bigr]
    \Bigr).
\end{aligned}
\end{equation}
Here $\exp(t H_{\mathrm{os}})$ and $\exp(t H_{\mathrm{\mathrm{os},\lambda}})$ denote the
Hamiltonian flows generated by $H_{\mathrm{os}}$ and
$H_{\mathrm{os},\lambda} = \tfrac12 p^{2} + \tfrac12 \lambda^{-2} x^{2}$,
respectively,
and for $i=1,2$, 
$\mathrm{pr}_i : \mathbb{R}^{d} \times \mathbb{R}^{d} \to \mathbb{R}^{d}$
denotes the projection onto the $i$-th component.
We first consider the behavior of $(x(\kappa),\xi(\kappa))$ 
for short times in $\kappa$, that is, for $\kappa$ in a small interval $[0,\delta]$.
\begin{proposition}\label{202512051157}
Assume that $(s,y,\eta)$  is non-trapping in the sense of
Definition~\ref{25080316}. 
Fix $0< \delta <1$, and let $\widetilde{\Omega} \subset \Omega$ be as in Proposition~\ref{250804}.
Then, for any $\theta \in (0,\delta]$ and 
$(\tilde{s},\tilde{y},\tilde{\eta}) \in \widetilde{\Omega}$,
the following limit holds:
\begin{equation*}
    \lim_{\lambda \to \infty}
    \exp \bigl(\tilde{s}(\theta\lambda)H_{\mathrm{os},\lambda}\bigr)
    \bigl(z_{\lambda}(\theta\lambda,\tilde{s},\tilde{y},\tilde{\eta}),\,\gamma_{\lambda}(\theta\lambda,\tilde{s},\tilde{y},\tilde{\eta})\bigr)
    =
    \bigl(x_{+}(\tilde{s},\tilde{y},\tilde{\eta}),
          \xi_{+}(\tilde{s},\tilde{y},\tilde{\eta})\bigr),
\end{equation*}
where $x_{+}$ and $\xi_{+}$ are those defined in Proposition~\ref{250123}.
Moreover, this convergence is uniform on $\widetilde{\Omega}$ and $\theta \in(0,\delta]$.
\end{proposition}
\begin{proof}
We denote
\begin{equation}\label{202512041002}
\begin{split}
(\widehat{z}(\kappa), \widehat{\gamma}(\kappa))
&= (\widehat{z}(\kappa,\lambda,\tilde{s},\tilde{y},\tilde{\eta}),
    \widehat{\gamma}(\kappa,\lambda,\tilde{s},\tilde{y},\tilde{\eta})) \\
&= \exp\bigl(\tilde{s}(\kappa\lambda) H_{\mathrm{os},\lambda}\bigr)
   (z_{\lambda}(\kappa,\tilde{s},\tilde{y},\tilde{\eta}),
    \gamma_{\lambda}(\kappa,\tilde{s},\tilde{y},\tilde{\eta})) \\[4pt]
&= \Bigl(
      \cos\left(\tfrac{\tilde{s}\kappa}{\lambda}\right) z_{\lambda}(\kappa)
      + \lambda \sin\left(\tfrac{\tilde{s}\kappa}{\lambda}\right) \gamma_{\lambda}(\kappa),
\\[-2pt]
&\qquad\quad
      -\tfrac{1}{\lambda} \sin\left(\tfrac{\tilde{s}\kappa}{\lambda}\right) z_{\lambda}(\kappa)
      + \cos\left(\tfrac{\tilde{s}\kappa}{\lambda}\right) \gamma_{\lambda}(\kappa)
   \Bigr).
\end{split}
\end{equation}
We first show that
\[
\tfrac{\mathrm d}{\mathrm d\kappa}\widehat{z}(\kappa) = O(\langle \kappa \rangle^{-1-\epsilon}),
\qquad
\tfrac{\mathrm d}{\mathrm d\kappa}\widehat{\gamma}(\kappa) = O(\langle \kappa \rangle^{-2-\epsilon})
\]for $\kappa \in [0,\delta\lambda]$.
By Assumption \ref{250124} direct computations yield the following:
\begin{equation}\label{202512051139}
\begin{split}
\tfrac{\mathrm d}{\mathrm d\kappa}\widehat{z_{m}}(\kappa)
 &= \cos\left(\tfrac{\tilde{s}\kappa}{\lambda}\right)
    \bigl( (-\tilde{s})(a_{mj}((1-\lambda^{-1}\kappa)\tilde{s}, z_{\lambda})
        - \delta_{mj}) \gamma_{\lambda,j} \bigr) \\
 &\quad + \sin\left(\tfrac{\tilde{s}\kappa}{\lambda}\right)
    \bigl( \tilde{s}\lambda^{-1} (\partial_{m} V)
        ((1-\lambda^{-1}\kappa)\tilde{s}, z_{\lambda}) \bigr) \\
 &\quad + \sin\left(\tfrac{\tilde{s}\kappa}{\lambda}\right)
    \bigl( \tfrac{1}{2} \tilde{s} \lambda
        (\partial_{m} a_{ij})((1-\lambda^{-1}\kappa)\tilde{s}, z_{\lambda})
        \gamma_{\lambda,i} \gamma_{\lambda,j} \bigr) \\
 &= O\bigl( \langle z_{\lambda} \rangle^{-1-\epsilon}
          + \lambda^{-1} \langle z_{\lambda} \rangle^{-\epsilon}
          + \lambda \langle z_{\lambda} \rangle^{-2-\epsilon} \bigr), \\[6pt]
\tfrac{\mathrm d}{\mathrm d\kappa}\widehat{\gamma_{m}}(\kappa)
 &= -\lambda^{-1} \sin\left(\tfrac{\tilde{s}\kappa}{\lambda}\right)
    \bigl( (-\tilde{s})(a_{mj}((1-\lambda^{-1}\kappa)\tilde{s}, z_{\lambda})
        - \delta_{mj}) \gamma_{\lambda,j} \bigr) \\
 &\quad + \cos\left(\tfrac{\tilde{s}\kappa}{\lambda}\right)
    \bigl( \tfrac{1}{2} \tilde{s} (\partial_{m} a_{ij})
        ((1-\lambda^{-1}\kappa)\tilde{s}, z_{\lambda})
        \gamma_{\lambda,i} \gamma_{\lambda,j} \bigr) \\
 &\quad + \cos\left(\tfrac{\tilde{s}\kappa}{\lambda}\right)
    \bigl( \tilde{s}\lambda^{-2} (\partial_{m} V)
        ((1-\lambda^{-1}\kappa)\tilde{s}, z_{\lambda}) \bigr) \\
 &= O\bigl( \lambda^{-1}\langle z_{\lambda} \rangle^{-1-\epsilon}
          + \langle z_{\lambda} \rangle^{-2-\epsilon}
          + \langle z_{\lambda} \rangle^{-2-\epsilon} \bigr).
\end{split}
\end{equation}
From Proposition \ref{250804}, there exists a constant $C_{1}>0$ such that
\begin{equation}\label{202512031051}
\left| \tfrac{\mathrm d}{\mathrm d\kappa}\widehat{z_{m}}(\kappa) \right|
 \le C_{1} \langle \kappa \rangle^{-1-\epsilon},
\qquad
\left| \tfrac{\mathrm d}{\mathrm d\kappa}\widehat{\gamma_{m}}(\kappa) \right|
 \le C_{1} \langle \kappa \rangle^{-2-\epsilon}
\end{equation}
for any $\lambda \ge \lambda_{0}$,
$(\tilde{s},\tilde{y},\tilde{\eta}) \in \widetilde{\Omega}$
and $\kappa \in [0,\delta\lambda]$.
Moreover, for each $\kappa \in (0,\delta\lambda]$ we have
\begin{equation}\label{202512031052}
\begin{split}
\lim_{\lambda\to\infty} \tfrac{\mathrm d}{\mathrm d\kappa}\widehat{z}_{m}(\kappa)
 &= (-\tilde{s})(a_{mj}(\tilde{s},
    x((1-\kappa)\tilde{s},\tilde{s},\tilde{y},\tilde{\eta}))
    - \delta_{mj})
    \xi_{j}((1-\kappa)\tilde{s},\tilde{s},\tilde{y},\tilde{\eta}) \\
 &\quad + \tfrac{\tilde{s}^{2}\kappa}{2}
    (\partial_{m} a_{ij})(\tilde{s},
     x((1-\kappa)\tilde{s},\tilde{s},\tilde{y},\tilde{\eta}))
    \xi_{i}((1-\kappa)\tilde{s}) \xi_{j}((1-\kappa)\tilde{s}), \\[4pt]
\lim_{\lambda\to\infty} \tfrac{\mathrm d}{\mathrm d\kappa}\widehat{\gamma}_{m}(\kappa)
 &= \tfrac{\tilde{s}}{2} (\partial_{m} a_{ij})
    (\tilde{s}, x((1-\kappa)\tilde{s},\tilde{s},\tilde{y},\tilde{\eta}))
    \xi_{i}((1-\kappa)\tilde{s})
    \xi_{j}((1-\kappa)\tilde{s}).
\end{split}
\end{equation}
By \eqref{202512031051}, \eqref{202512031052} and the dominated convergence theorem,
for each $\theta \in [0,\delta]$ we conclude
\begin{equation*}
\begin{split}
\lim_{\lambda\to\infty} \widehat{z}(\theta\lambda)
 &= \tilde{y}
    + \lim_{\lambda\to\infty}
      \int_{0}^{\theta\lambda}
      \tfrac{\mathrm d}{\mathrm d\kappa}\widehat{z}(\kappa)\,\mathrm d\kappa \\
 &= \tilde{y}
    + \int_{0}^{\infty}
      \lim_{\lambda\to\infty}
      \tfrac{\mathrm d}{\mathrm d\kappa}\widehat{z}(\kappa)\,\mathrm d\kappa
  = x_{+}(\tilde{s},\tilde{y},\tilde{\eta}), \\[6pt]
\lim_{\lambda\to\infty} \widehat{\gamma}(\theta\lambda)
 &= \tilde{\eta}
    + \lim_{\lambda\to\infty}
      \int_{0}^{\theta\lambda}
      \tfrac{\mathrm d}{\mathrm d\kappa}\widehat{\gamma}(\kappa)\,\mathrm d\kappa \\
 &= \tilde{\eta}
    + \int_{0}^{\infty}
      \lim_{\lambda\to\infty}
      \tfrac{\mathrm d}{\mathrm d\kappa}\widehat{\gamma}(\kappa)\,\mathrm d\kappa
  = \xi_{+}(\tilde{s},\tilde{y},\tilde{\eta}).
\end{split}
\end{equation*}
Here we use that $x(\kappa)$ and $\xi(\kappa)$ satisfy \eqref{202511302324}.
The uniformity of the convergence with respect to $\tilde{\Omega}$ and $\theta$ follows from the fact that
\eqref{202512031051} holds uniformly on $\tilde{\Omega}$ for $\kappa \in [0,\delta \lambda]$.
This completes the proof of the assertion.
\end{proof}
We next study the behavior of $(x(\kappa),\xi(\kappa))$ 
for $\kappa$ close to 1.
\begin{proposition}\label{202512071951}
Assume that $(s,y,\eta)$ with $|s|<\pi$ is non-trapping in the sense of
Definition~\ref{25080316}. 
Fix $0< \delta <1$, and let $\widetilde{\Omega} \subset \Omega$ be as in Proposition~\ref{250804}.
Then, for any $\theta \in [\delta,1]$ and 
$(\tilde{s},\tilde{y},\tilde{\eta}) \in \widetilde{\Omega}$,
the following identity holds:
\begin{equation*}
    \lim_{\lambda \to \infty}
    \exp \bigl(s(\theta\lambda)H_{\mathrm{os},\lambda}\bigr)
    \bigl(z_{\lambda}(\theta\lambda,\tilde{s},\tilde{y},\tilde{\eta}),\gamma_{\lambda}(\theta\lambda,\tilde{s},\tilde{y},\tilde{\eta})\bigr)
    =
    \bigl(x_{+}(\tilde{s},\tilde{y},\tilde{\eta}),
          \xi_{+}(\tilde{s},\tilde{y},\tilde{\eta})\bigr),
\end{equation*}
where 
$x_{+}$ and $\xi_{+}$ are those defined in Proposition~\ref{250123}.
Moreover, this convergence is uniform on $\widetilde{\Omega}$ and $\theta \in [\delta,1]$.
\end{proposition}
\begin{proof}
  Let $0<\epsilon_{0}<\tfrac{1}{2} \min_{\widetilde{\Omega}}|\xi_{+}|$.  
  We show, by contradiction, that for any $\epsilon>0$ there exists $\lambda_{0}>0$ such that, for any $\lambda \ge \lambda_{0}$, we have
  \begin{equation}\label{202512051211}
\max_{\substack{
  \theta \in [\delta,1] \\
  (\tilde{s},\tilde{y},\tilde{\eta}) \in \tilde{\Omega}
}}
\left(
  \left|
    \widehat{z}(\theta\lambda,\tilde{s},\tilde{y},\tilde{\eta})
    - x_{+}(\tilde{s},\tilde{y},\tilde{\eta})
  \right|,
  \left|
    \widehat{\gamma}(\theta\lambda,\tilde{s},\tilde{y},\tilde{\eta})
    - \xi_{+}(\tilde{s},\tilde{y},\tilde{\eta})
  \right|
\right)
< \epsilon .
\end{equation}
  Here $\widehat{z}(\kappa)$ and $\widehat{\gamma}(\kappa)$ are defined by \eqref{202512041002}. 
  We first note that
  \begin{equation*}
    \begin{split}
      z_{\lambda}(\kappa)
      &= \cos(\tfrac{\tilde{s}\kappa}{\lambda})\widehat{z}(\kappa)
        - \lambda \sin(\tfrac{\tilde{s}\kappa}{\lambda})\widehat{\gamma}(\kappa), \\
      \gamma_{\lambda}(\kappa)
      &= \lambda^{-1} \sin(\tfrac{\tilde{s}\kappa}{\lambda}) \widehat{z}(\kappa)
         + \cos(\tfrac{\tilde{s}\kappa}{\lambda}) \widehat{\gamma}(\kappa).
    \end{split}
  \end{equation*}
  For the moment, suppose that for any $\kappa \in [\delta \lambda, \tilde{\theta} \lambda]$ and $(\tilde{s},\tilde{y},\tilde{\eta}) \in \tilde{\Omega}$ we have  
  $|\widehat{z}(\kappa,\tilde{s},\tilde{y},\tilde{\eta})-x_{+}(\tilde{s},\tilde{y},\tilde{\eta})|<\epsilon_{0}$ and  
  $|\widehat{\gamma}(\kappa,\tilde{s},\tilde{y},\tilde{\eta})-\xi_{+}(\tilde{s},\tilde{y},\tilde{\eta})|<\epsilon_{0}$ 
  for some $\tilde{\theta} \in (\delta,1]$. 
  Then
  \begin{equation*}
    |z_{\lambda}(\kappa,\tilde{s},\tilde{y},\tilde{\eta})|
    \ge |\lambda \sin(\tfrac{\tilde{s}\kappa}{\lambda})|\bigl(|\xi_{+}(\tilde{s},\tilde{y},\tilde{\eta})| -\epsilon_{0}\bigr)
      -\bigl(|x_{+}(\tilde{s},\tilde{y},\tilde{\eta})|+\epsilon_{0}\bigr).
  \end{equation*}
  Hence, by choosing $\lambda_{0}$ sufficiently large and $\lambda \ge \lambda_{0}$, we obtain
  \begin{equation*}
    |z_{\lambda}(\kappa)| \ge c_{1} \lambda
  \end{equation*}
  for any $\kappa \in [\delta\lambda, \tilde{\theta} \lambda]$ and some constant $c_{1}>0$ depending only on $x_{+}$ and $\xi_{+}$.  
  Here we use the fact that $|\tilde{s}|<\pi$.
  Then, using formula \eqref{202512051139}, we obtain
  \begin{equation}\label{202512051513}
    \left|\tfrac{\mathrm d}{\mathrm d\kappa} \widehat{z}(\kappa)\right| \le C_{1} \lambda^{-1-\epsilon},
    \qquad
    \left|\tfrac{\mathrm d}{\mathrm d\kappa} \widehat{\gamma}(\kappa)\right| \le C_{1} \lambda^{-2-\epsilon}
  \end{equation}
  for any $\kappa \in [\delta\lambda , \tilde{\theta} \lambda]$ and some constant $C_{1}>0$.
  Assume the negation of \eqref{202512051211}. Then there exist a sequence $\lambda_{n} \nearrow \infty$ and  
  $0<\epsilon_{1}<\epsilon_{0}$ such that
  \begin{equation}\label{202512051221}
    \max_{\substack{
  \theta \in [\delta,1] \\
  (\tilde{s},\tilde{y},\tilde{\eta}) \in \tilde{\Omega}
}}
    \left(
      |\widehat{z}(\theta \lambda_{n},\tilde{s},\tilde{y},\tilde{\eta}) - x_{+}(\tilde{s},\tilde{y},\tilde{\eta})|,
      |\widehat{\gamma}(\theta \lambda_{n},\tilde{s},\tilde{y},\tilde{\eta}) - \xi_{+}(\tilde{s},\tilde{y},\tilde{\eta})|
    \right)
    \ge \epsilon_{1}
  \end{equation}
  for all $n$.
  From Proposition \ref{202512051157}, we obtain
  \begin{equation}\label{202512051222}
    \max_{(\tilde{s},\tilde{y},\tilde{\eta}) \in \tilde{\Omega}}
    \left(
      |\widehat{z}(\delta \lambda_{n},\tilde{s},\tilde{y},\tilde{\eta}) - x_{+}(\tilde{s},\tilde{y},\tilde{\eta})|,
      |\widehat{\gamma}(\delta \lambda_{n},\tilde{s},\tilde{y},\tilde{\eta}) - \xi_{+}(\tilde{s},\tilde{y},\tilde{\eta})|
    \right)
    < \tfrac{\epsilon_{1}}{2}
  \end{equation}
  for sufficiently large $n$.
  From \eqref{202512051221} and \eqref{202512051222}, there exists  
  $\theta_{n} \in (\delta,1] $  
  such that
  \begin{equation*}
    \max_{(\tilde{s},\tilde{y},\tilde{\eta}) \in \tilde{\Omega}}
    \left(
      |\widehat{z}(\theta_{n} \lambda_{n},\tilde{s},\tilde{y},\tilde{\eta}) - x_{+}(\tilde{s},\tilde{y},\tilde{\eta})|,
      |\widehat{\gamma}(\theta_{n} \lambda_{n},\tilde{s},\tilde{y},\tilde{\eta}) - \xi_{+}(\tilde{s},\tilde{y},\tilde{\eta})|
    \right)
    = \epsilon_{1},
  \end{equation*}
  and
  \begin{equation*}
    \max_{(\tilde{s},\tilde{y},\tilde{\eta}) \in \tilde{\Omega}}
    \left(
      |\widehat{z}(\tilde{\kappa},\tilde{s},\tilde{y},\tilde{\eta}) - x_{+}(\tilde{s},\tilde{y},\tilde{\eta})|,
      |\widehat{\gamma}(\tilde{\kappa},\tilde{s},\tilde{y},\tilde{\eta}) - \xi_{+}(\tilde{s},\tilde{y},\tilde{\eta})|
    \right)
    \le \epsilon_{1}
  \end{equation*}
  for any $\tilde{\kappa} \in [\delta \lambda_{n}, \theta_{n} \lambda_{n}]$.
  From \eqref{202512051513} and \eqref{202512051222}, we have
  \begin{equation*}
    \begin{split}
      |\widehat{z}(\theta_{n} \lambda_{n},\tilde{s},\tilde{y},\tilde{\eta}) - x_{+}(\tilde{s},\tilde{y},\tilde{\eta})|
      &= \left| \int_{\delta \lambda_{n}}^{\theta_{n} \lambda_{n}} \tfrac{\mathrm d}{\mathrm d\kappa} \widehat{z}(\kappa)\,\mathrm d\kappa
              + \widehat{z}(\delta \lambda_{n}) - x_{+}(\tilde{s},\tilde{y},\tilde{\eta}) \right| \\
      &\le C_{2} (1-\delta) \lambda_{n}^{-\epsilon} + \tfrac{\epsilon_{1}}{2}, \\
      | \widehat{\gamma}(\theta_{n}\lambda_{n},\tilde{s},\tilde{y},\tilde{\eta}) - \xi_{+}(\tilde{s},\tilde{y},\tilde{\eta})|
      &= \left| \int_{\delta \lambda_{n}}^{\theta_{n} \lambda_{n}} \tfrac{\mathrm d}{\mathrm d\kappa} \widehat{\gamma}(\kappa)\,\mathrm d\kappa
              + \widehat{\gamma}(\delta \lambda_{n}) - \xi_{+}(\tilde{s},\tilde{y},\tilde{\eta}) \right| \\
      &\le C_{2} (1-\delta) \lambda_{n}^{-1-\epsilon} + \tfrac{\epsilon_{1}}{2},
    \end{split}
  \end{equation*}
  for some constant $C_{2}>0$ depending only on $x_{+}$ and $\xi_{+}$.
  Hence,
  \begin{equation*}
    \begin{split}
      \epsilon_{1}
      &= \max_{(\tilde{s},\tilde{y},\tilde{\eta}) \in \tilde{\Omega}}
         \left(
           |\widehat{z}(\theta_{n} \lambda_{n},\tilde{s},\tilde{y},\tilde{\eta}) - x_{+}(\tilde{s},\tilde{y},\tilde{\eta})|,
           |\widehat{\gamma}(\theta_{n} \lambda_{n},\tilde{s},\tilde{y},\tilde{\eta}) - \xi_{+}(\tilde{s},\tilde{y},\tilde{\eta})|
         \right) \\
      &\le \tfrac{\epsilon_{1}}{2}
         + C_{2}(1-\delta)\bigl(\lambda_{n}^{-\epsilon} + \lambda_{n}^{-1-\epsilon}\bigr)
    \end{split}
  \end{equation*}
  for all $n$.  
  By choosing $n$ sufficiently large, this yields a contradiction.  
  Thus the assertion follows.
\end{proof}

\subsection{Conclusions}

We denote the flow generated by the system \eqref{25042320}--\eqref{25042323} 
with initial data \eqref{25042324} by 
$F_l\colon \mathbb R\times \mathbb R^{2(1+d)}\to \mathbb R^{2(1+d)}$. 
We are interested in the asymptotic behavior of 
\[F_l(\kappa,s,y,\lambda^2\sigma,\lambda\eta)\]
as $\lambda\to\infty$. For that we introduce the scaling operator on $\mathbb R^{2(1+d)}$:
\[
\Theta_\lambda\colon \mathbb R^{2(1+d)}\to\mathbb R^{2(1+d)},\ \ (s,y,\sigma,\eta)\mapsto (s,y,\lambda^2\sigma,\lambda\eta)
,
\]
and consider the limit of the composition
\[
\lim_{\lambda\to\infty}\Theta_\lambda^{-1}F_l\Theta_\lambda. 
\]
\begin{proposition}\label{202512091350}
  The following three properties hold.
  \leavevmode
\begin{enumerate}

\item
For any $(s,y,\eta)\in \Omega$ with $|s|<\pi$ and any $\sigma\in\mathbb{R}$, 
there exists a neighborhood $\widehat{\Omega}$ of $(s,y,\sigma,\eta)$ such that 
for any $(\hat{s},\hat{y},\hat{\sigma},\hat{\eta}) \in \widehat{\Omega}$
\[
\lim_{\substack{\lambda\to\infty \\[2pt] \kappa\in(0,1]}}
\Theta_\lambda^{-1} F_l \Theta_\lambda(\kappa,\hat{s},\hat{y},\hat{\sigma},\hat{\eta})
=
\bigl(\hat{s},\; x_+(\hat{s},\hat{y},\hat{\eta}),\; \Sigma(\hat{s},\hat{y},\hat{\sigma},\hat{\eta}),\; \xi_+(\hat{s},\hat{y},\hat{\eta})\bigr),
\]
where $\Sigma(\hat{s},\hat{y},\hat{\sigma},\hat{\eta})$ is given by
\[
\Sigma(\hat{s},\hat{y},\hat{\sigma},\hat{\eta})
=
\hat{\sigma}
+ \tfrac12 a_{ij}(\hat{s},\hat{y}) \hat{\eta}_i \hat{\eta}_j
- \tfrac12 |\xi_+(\hat{s},\hat{y},\hat{\eta})|^2 .
\]

\item 
Moreover, for any  
$\alpha \in \mathbb{N}_0^{2(1+d)}$, $(\hat{s},\hat{y},\hat{\sigma},\hat{\eta}) \in \widehat{\Omega}$ and $\kappa \in [0,1]$,
\[
\begin{aligned}
\partial^\alpha (\Theta_\lambda^{-1} F_l \Theta_\lambda)
      (\kappa,\hat{s},\hat{y},\hat{\sigma},\hat{\eta})
\rightarrow
\bigl(
  \partial^\alpha \hat{s},
  \partial^\alpha x_{+}(\hat{s},\hat{y},\hat{\eta}),
  \partial^\alpha \Sigma(\hat{s},\hat{y},\hat{\eta}),
  \partial^\alpha \xi_{+}(\hat{s},\hat{y},\hat{\eta})
\bigr)
\end{aligned}
\]
as $\lambda \to \infty$, and the convergence is uniform on 
$\widehat{\Omega}$ and $\kappa \in [0,1]$.
\item
For any $(s,y,\eta)\in \Omega$ with $|s|<\pi$, $\sigma \in \mathbb{R}$
and $\kappa\in[0,1]$, 
the mapping
\[
\lim_{\substack{\lambda\to\infty }}
(\Theta_\lambda^{-1} F_l \Theta_\lambda)(\kappa,\cdot ,\cdot,\cdot,\cdot)
\]
is a local diffeomorphism in a neighborhood of $(s,y,\sigma,\eta)$.

\end{enumerate}

\end{proposition}
  \begin{proof}
\noindent\text{1.}\;
We set $\widehat{\Omega}
= \{(\hat{s},\hat{y},\hat{\sigma},\hat{\eta}) \in \mathbb{R}^{2(1+d)}
\mid (\hat{s},\hat{y},\hat{\eta}) \in \tilde{\Omega},\; |\hat{\sigma}-\sigma| < 1 \}$.
Here $\tilde{\Omega}$ is as in Proposition~\ref{250804}.
Let $(\hat{s},\hat{y},\hat{\sigma},\hat{\eta}) \in \widehat{\Omega}$ be fixed.
From Propositions \ref{202512051157} and \ref{202512071951}, together with 
\eqref{202512071953}, we obtain
\[
  \bigl(x(\kappa,\hat{s},\hat{y},\hat{\eta}),\, \lambda^{-1}\xi(\kappa,\hat{s},\hat{y},\hat{\eta})\bigr)
  \longrightarrow 
  \bigl(x_{+}(\hat{s},\hat{y},\hat{\eta}),\, \xi_{+}(\hat{s},\hat{y},\hat{\eta})\bigr)
  \quad \text{as } \lambda \to \infty.
\]
Moreover, by \eqref{2511261553} we have 
\begin{equation}\label{202512072113}
  \begin{split} 
   \lambda^{-2}\tau(\kappa,\hat{s},\hat{y},\hat{\eta}) 
  &=\hat{\sigma}
    +\tfrac{1}{2} a_{ij}(\hat{s},\hat{y})\hat{\eta}_{i}\hat{\eta}_{j}
    -\tfrac{\kappa}{2} |\gamma_{\lambda}(\lambda \kappa)|^{2}
    -\tfrac{1}{2}\lambda^{-2}|z_{\lambda}(\lambda \kappa)|^{2}  \\
  &\quad -(1-\kappa)\Bigl(
       \tfrac12 a_{ij}((1-\kappa)\hat{s},z_{\lambda}(\lambda \kappa))
       \gamma_{\lambda,i}(\lambda\kappa)
       \gamma_{\lambda,j}(\lambda\kappa)  \\
  &\quad\qquad\qquad
       + \lambda^{-2}V((1-\kappa)\hat{s},z_{\lambda}(\lambda \kappa))
     \Bigr) \\  
  &\quad + \lambda^{-2} V(\hat{s},\hat{y})
    +\tfrac{1}{2}\lambda^{-2}|\hat{y}|^{2} \\
  &\longrightarrow 
  \Sigma(\hat{s},\hat{y},\hat{\sigma},\hat{\eta})
  \quad \text{as } \lambda \to \infty.
  \end{split}
\end{equation}
Hence the assertion follows.

\smallskip

\noindent\text{2.}\;
By \eqref{25042320} and \eqref{202512072113} 
it suffices to consider 
$\partial^{\alpha} x(\kappa,\hat{s},\hat{y},\hat{\eta})$ and 
$\lambda^{-1}\partial^{\alpha} \xi(\kappa,\hat{s},\hat{y},\hat{\eta})$, 
which are uniformly convergent on $\widehat{\Omega}$ and $\kappa \in [0,1]$.
The functions $x(\kappa)$ and $\xi(\kappa)$ satisfy 
\eqref{25042321} and \eqref{25042323}, and 
$\partial^{\alpha} x$ and $\lambda^{-1}\partial^{\alpha} \xi$ satisfy 
similar first order ODEs whose coefficients are integrable in $\kappa$.
Hence, using induction on the order $|\alpha|$ of derivatives, 
the assertion follows.

\smallskip

\noindent\text{3.}\;
First, we consider the Jacobian matrix of 
$
(\Theta_{\lambda}^{-1} F_l \Theta_{\lambda})
$
with respect to $(\hat{s},\hat{y},\hat{\sigma},\hat{\eta})$.
It is invertible if and only if that of 
$
(x(\kappa), \lambda^{-1}\xi(\kappa))
$
with respect to $(\hat{y},\hat{\eta})$ is invertible.
By \eqref{25042321} and \eqref{25042323}, the Jacobian matrix for $x$ and $\lambda^{-1}\xi$
is invertible near $(s,y,\sigma,\eta)$ for sufficiently large $\lambda$.
Hence the Jacobian matrix of 
$
(\Theta_{\lambda}^{-1} F_l \Theta_{\lambda})
$
is also invertible.
Moreover, by the assertion~2, the Jacobian matrix of 
$
\Theta_{\lambda}^{-1} F_l \Theta_{\lambda}
$
converges uniformly to the Jacobian matrix of
$
\lim_{\lambda\to\infty} \bigl(\Theta_{\lambda}^{-1} F_l \Theta_{\lambda}\bigr).
$
Therefore, the assertion follows.
\end{proof}
\section{Proof}
\subsection{Construction of symbol}

Here let us construct a symbol $b$ defining the operator \eqref{25062214} 
that satisfies the equation \eqref{25042517}. 
We are going to construct a symbol $b$ of the form 
\[
b\sim \sum_{j=0}^\infty h^{\epsilon j}b_j.
\]
By the pseudodifferential symbol calculus, the principal symbol $b_{0}$ should 
satisfy the transport equation
\begin{equation}\label{25062221}
\partial_\kappa b_{0}
+ (\partial_\tau l)\,\partial_t b_{0}
+ (\partial_\xi l)\,\partial_x b_{0}
- h^{2}(\partial_t l)\,\partial_\tau b_{0}
- h(\partial_x l)\,\partial_\xi b_{0}
= 0,
\end{equation}
with $b_{0}(0,t,x,\tau,\xi)=a_{0}(t,x,\tau,\xi)$,
where $l$ is from \eqref{250619} and $a_{0}$ is defined in Section~\ref{202512081453}.
By using the flow $F_l$ we can rewrite the equation \eqref{25062221} as 
\[
\tfrac{\mathrm d}{\mathrm d\kappa}\bigl\{b_0(\kappa,\Theta_hF_l)\bigr\}=0,
\]
so that  
\[
b_0(\kappa,\Theta_hF_l(\kappa,t,x,\tau,\xi))=b_0(0,\Theta_h(t,x,\tau,\xi))=a_0(t,x,h^2\tau,h\xi)
.
\]
Hence 
\begin{equation*}
b_0(\kappa,t,x,\tau,\xi)=a_0(\Theta_hF_l^{-1}(\kappa)\Theta_h^{-1}(t,x,\tau,\xi))
,
\label{25062222}
\end{equation*}
where $F_{l}(\kappa)$ denotes $F_{l}(\kappa,\cdot,\cdot,\cdot,\cdot)$ for brevity.

\begin{proposition}\label{202512051746}
  There exists $b(\kappa,\cdot,\cdot,\cdot,\cdot) \in C^{\infty}_{\mathrm c}(\mathbb{R}^{2(1+d)}) $ for $\kappa \in [0,1]$ such that
  \begin{enumerate}
\item
  $b(0,t,x,\tau,\xi) = a_0(t,x,\tau,\xi).$
\item
$b(\kappa,\cdot,\cdot,\cdot,\cdot)$ is supported in $(\Theta_{h} F_{l}(\kappa)  \Theta_{h}^{-1})(\operatorname{supp} a_{0})$.
\item 
For any $\alpha \in \mathbb{N}_0^{2(1+d)}$, there exist constants
$C_\alpha>0$ and $h_\alpha>0$ such that
\[
\bigl|\partial^\alpha b(\kappa,t,x,\tau,\xi)\bigr|
  \le C_\alpha
\]
for all $\kappa\in[0,1]$, $(t,x,\tau,\xi)\in\mathbb{R}^{2(1+d)}$ and $h \in (0,h_{\alpha}]$.
\item 
The principal symbol of $b(\kappa,t,x,\tau,\xi)$ is $b_{0}(\kappa,t,x,\tau,\xi)$.
That is, for any $\alpha \in \mathbb{N}_0^{2(1+d)}$, there exist constants
$C_\alpha>0$ and $h_\alpha>0$ such that
\begin{equation*}
   \bigl|\partial^\alpha \bigl(b(\kappa,t,x,\tau,\xi)
   - b_{0}(\kappa,t,x,\tau,\xi)\bigr)\bigr|
  \le C_\alpha h^{\varepsilon}
\end{equation*}
for all $\kappa\in[0,1]$, $(t,x,\tau,\xi)\in\mathbb{R}^{2(1+d)}$
and $h \in (0,h_{\alpha}]$.
\item
If we set $B(\kappa) = b^{W}(\kappa,t,x,h^{2}p_{t},hp_{x})$, Then
\begin{equation*}
  \left\|\tfrac{\mathrm d}{\mathrm d\kappa}B(\kappa)+\mathrm i[L(\kappa),B(\kappa)] \right\|_{L^2_{t,x}} = O(h^{\infty}) 
\end{equation*}
as $h \rightarrow 0$, uniformly in $\kappa \in [0,1]$.
\end{enumerate}
\end{proposition}

\begin{proof}
If necessary, we may assume that $\operatorname{supp} a_{0} \subset \widehat{\Omega}$,
where $\widehat{\Omega}$ is the neighborhood specified in Proposition~\ref{202512091350}.
Then, by Proposition~\ref{202512091350}, for any multi-index
$\alpha \in \mathbb{N}^{2(1+d)}_{0}$ there exists $h_{0}>0$ such that for any
$\kappa \in [0,1]$, $(t,x,\tau,\xi) \in \widehat{\Omega}$, and $h \in (0,h_{0}]$,
we have
\begin{equation*}
  \bigl|\partial^{\alpha} b_{0}(\kappa,t,x,\tau,\xi)\bigr|
  \le C_{\alpha}.
\end{equation*}
We construct the symbol $b$ by an iterative procedure as follows.
We choose $\Gamma$ sufficiently large so that
$b_{0}(\kappa,t,x,\tau,\xi)=0$ for all $\kappa \in [0,1]$,
$(\tau,\xi) \in \mathbb{R}^{1+d}$, and $(t,x) \notin \Gamma$.
Define
\begin{equation*}
  r_{0}^{W}(\kappa,t,x,p_{t},p_{x})
  = \tfrac{\partial}{\partial \kappa}
    b^{W}_{0}(\kappa,t,x,h^{2}p_{t},hp_{x})
    + i\bigl[L(\kappa),
      b^{W}_{0}(\kappa,t,x,h^{2}p_{t},hp_{x})\bigr].
\end{equation*}
Then, by the asymptotic expansion formula, for any
$\gamma = (\alpha,\tilde{\alpha},\beta,\tilde{\beta}) \in \mathbb{N}^{2(1+d)}_{0}$ there exist constants
$C_{\gamma}$ and $h_{\gamma} > 0$ such that
\begin{equation*}
  \bigl|\partial^{\alpha}_{t} \partial^{\tilde{\alpha}}_{x} \partial^{\beta}_{\tau} \partial^{\tilde{\beta}}_{\xi}  r_{0}(\kappa,t,x,\tau,\xi)\bigr|
  \le C_{\gamma} h^{\epsilon+\tilde{|\beta|}+ 2|\beta|}
\end{equation*}
for all $\kappa \in [0,1]$, $(t,x) \in \Gamma$,
$(\tau,\xi) \in \mathbb{R}^{1+d}$, and $h \in (0,h_{\gamma}]$.
Moreover, $r_{0}$ is essentially supported in
$\operatorname{supp} b_{0}(\kappa,\cdot,\cdot,\cdot,\cdot)$,
that is, modulo $O(h^{\infty})$ terms.
Next, we solve the transport equation
\begin{equation*}
  \partial_\kappa (h^{\epsilon} b_{1})
+ (\partial_\tau l)\,\partial_t (h^{\epsilon} b_{1})
+ (\partial_\xi l)\,\partial_x (h^{\epsilon} b_{1})
- h^{2}(\partial_t l)\,\partial_\tau (h^{\epsilon} b_{1})
- h(\partial_x l)\,\partial_\xi (h^{\epsilon} b_{1})
  = - r_{0},
\end{equation*}
with the initial condition $b_{1}(0,t,x,\tau,\xi)=0$.
Then, for any multi-index $\alpha$, there exist 
$C_{\alpha} > 0$ and $h_{\alpha} > 0$ such that
\begin{equation*}
  \bigl|\partial^{\alpha} b_{1}(\kappa,t,x,\tau,\xi)\bigr|
  \le C_{\alpha}
\end{equation*}
for all $\kappa \in [0,1]$, $(t,x) \in \Gamma$,
$(\tau,\xi) \in \mathbb{R}^{1+d}$, and $h \in (0,h_{\alpha}]$.
Moreover, $b_{1}$ is also essentially supported in
$\operatorname{supp} b_{0}$.
If we set $r_{1}^{W}$ 
\begin{equation*}
  \begin{split}
    r^{W}_{1}(\kappa,t,x,p_{t},p_{x})&= 
   \tfrac{\partial}{\partial \kappa}
    h^{\epsilon} b^{W}_{1}(\kappa,t,x,h^{2}p_{t},hp_{x})
    + i\bigl[L(\kappa),
      h^{\epsilon} b^{W}_{1}(\kappa,t,x,h^{2}p_{t},hp_{x})\bigr] \\
    &\quad + r^{W}_{0}(\kappa,t,x,p_{t},p_{x}),
  \end{split}
\end{equation*}
Then $r_{1}$ satisfies, for any
$\gamma = (\alpha,\tilde{\alpha},\beta,\tilde{\beta}) \in \mathbb{N}^{2(1+d)}_{0}$,
\begin{equation*}
  \bigl|
    \partial^{\alpha}_{t}
    \partial^{\tilde{\alpha}}_{x}
    \partial^{\beta}_{\tau}
    \partial^{\tilde{\beta}}_{\xi}
    r_{1}(\kappa,t,x,\tau,\xi)
  \bigr|
  \le C_{\gamma}\, h^{2\epsilon +|\tilde{\beta}| +2|\beta|}
\end{equation*}
for all $\kappa \in [0,1]$, $(t,x) \in \Gamma$,
$(\tau,\xi) \in \mathbb{R}^{1+d}$, and sufficiently small $h>0$.
Iterating this procedure, we obtain symbols $b_{j}$, $j=2,3,\ldots$, such that
\begin{equation*}
  \bigl|\partial^{\alpha} b_{j}(\kappa,t,x,\tau,\xi)\bigr|
  \le C_{\alpha}
\end{equation*}
for all $\kappa \in [0,1]$, $(t,x) \in \Gamma$, and
$(\tau,\xi) \in \mathbb{R}^{1+d}$, and the corresponding remainders 
\[
\begin{split}
  r_{j}^{W} (\kappa,t,x,p_{t},p_{x})&= 
   \tfrac{\partial}{\partial \kappa}
    h^{\epsilon j} b^{W}_{j}(\kappa,t,x,h^{2}p_{t},hp_{x})
    + i\bigl[L(\kappa),
      h^{\epsilon j} b^{W}_{j}(\kappa,t,x,h^{2}p_{t},hp_{x})\bigr] \\
    &\quad + r^{W}_{j-1}(\kappa,t,x,p_{t},p_{x}),
\end{split}
\]
are $O(h^{\epsilon (j+1) })$ for any $j \ge 2$.
We then define
\[
  b(\kappa,t,x,\tau,\xi)
  \sim \sum_{j=0}^{\infty} h^{\epsilon j} b_{j}(\kappa,t,x,\tau,\xi),
\]
as an asymptotic sum as $h \to 0$.
We may choose $b(\kappa,t,x,\tau,\xi)$ to be supported in
$\operatorname{supp} b_{0}(\kappa,\cdot,\cdot,\cdot,\cdot)$,
since the error is $O(h^{\infty})$.
It is now straightforward to verify that
$b(\kappa,t,x,\tau,\xi)$ satisfies the required properties.
\end{proof}

\subsection{Proof of the main result}

Finally, we prove our main result.

\begin{proof}[Proof of Theorem \ref{250208}]
We take the function $b$ constructed in Proposition~\ref{202512051746}.
By Proposition~\ref{202512091350}, we have
$b(1,s,x_{+},-\tfrac12|\xi_{+}|^{2},\xi_{+}) \neq 0$
for sufficiently small $h>0$,
Moreover, its support is in a small
neighborhood of $(s,x_{+},-\tfrac12|\xi_{+}|^{2},\xi_{+})$.
Taking this fact into account, the proof can be carried out by following
the strategy described in Section~\ref{202512081453}.
\end{proof}


  


\subsubsection*{Funding}
KI was partially supported by JSPS KAKENHI Grant Number~JP23K03163.

\subsubsection*{Competing Interests}
The authors declare that they have no competing interests.

\subsubsection*{Data Availability}
Data sharing is not applicable to this article as no datasets were generated or analysed during the current study.

\end{document}